\documentclass[12pt]{amsproc}

\usepackage{setspace}
\usepackage{latexsym}
\usepackage{amsmath, amssymb, amsfonts, amsthm}
\usepackage{float}
\usepackage{color}
\usepackage{mathrsfs}
\usepackage{hyperref} 
\usepackage{enumerate}
\hypersetup{citecolor=red, linkcolor=blue, colorlinks=true}

\usepackage{bm}
\usepackage[margin=1in]{geometry}
\usepackage{url}

\newtheorem{thm}{Theorem}[section]
\newtheorem{lem}[thm]{Lemma}

\newtheorem{cor}[thm]{Corollary}

\newtheorem{conj}[thm]{Conjecture}

\renewcommand\le{\leqslant}
\renewcommand\ge{\geqslant}

\newcommand\Wr{{\rm \,wr\, }}

\newcommand\N{{\mathbb{N}}}
\newcommand\Aut{{\rm Aut}}

\newcommand\Sym{{\rm Sym}}
\newcommand\I{{\rm I}}

\newcommand{\calQ}{\mathcal{Q}}
\newcommand{\calP}{\mathcal{P}}
\newcommand{\calL}{\mathcal{L}}

\title{On prime order automorphisms of generalized quadrangles}

\author{Santana F. Afton}
\address{Department of Mathematics, College of William \& Mary, P.O. Box 8795, Williamsburg, VA 23187-8795, USA}
\curraddr{School of Mathematics, Georgia Institute of Technology, 686 Cherry Street, Atlanta, GA 30332-0160, USA}
\email{santana.afton@gatech.edu}

\author{Eric Swartz}
\address{Department of Mathematics, College of William \& Mary, P.O. Box 8795, Williamsburg, VA 23187-8795, USA}
\email{easwartz@wm.edu}

\begin{document}
 
\begin{abstract}
In this paper, we study prime order automorphisms of generalized quadrangles.  We show that, if $\mathcal{Q}$ is a thick generalized quadrangle of order $(s,t)$, where $s > t$ and $s+1$ is prime, and $\mathcal{Q}$ has an automorphism of order $s+1$, then 
\[ s \left\lceil \left\lceil \frac{t^2}{s+1}\right\rceil\left(\frac{s+1}{t} \right) \right\rceil \le t(s+t),\]
with a similar inequality holding in the dual case when $t > s$, $t+1$ is prime, and $\mathcal{Q}$ is a thick generalized quadrangle of order $(s,t)$ with an automorphism of order $t+1$.

In particular, if $s+1$ is prime and if there exists a natural number $n$ such that 
\[ \frac{t^2}{n+1} + t \le s + 1 < \frac{t^2}{n},\]
then a thick generalized quadrangle $\mathcal{Q}$ cannot have an automorphism of order $s+1$, and hence the automorphism group of $\mathcal{Q}$ cannot be transitive on points.  These results apply to numerous potential orders for which it is still unknown whether or not generalized quadrangles exist, showing that any examples would necessarily be somewhat asymmetric.  Finally, we are able to use the theory we have built up about prime order automorphisms of generalized quadrangles to show that the automorphism group of a potential generalized quadrangle of order $(4,12)$ must necessarily be intransitive on both points and lines.
\end{abstract}

\maketitle

\section{Introduction}

Following \cite{fgq}, a \textit{finite generalized quadrangle} $\calQ$ is an incidence structure $(\calP, \calL, \I)$, where $\calP$ is the set of \textit{points}, $\calL$ is the set of \textit{lines} (which is disjoint from $\calP$), and $\I$ is a symmetric point-line incidence relation satisfying the following axioms: 

\begin{description}
 \item[Point-line incidence] Each point is incident with $t+1$ lines and each line is incident with $s+1$ points, where $s,t \in \N$, and two distinct points (respectively, lines) are mutually incident with at most one line (respectively, point).
 \item[GQ Axiom] Given a point $P$ and a line $\ell$ not incident with $P$, there is a unique pair $(P^\prime, \ell^\prime) \in \calP \times \calL$ such that $P \;\I\; \ell^\prime \;\I\; P^\prime\; \I \;\ell$. 
\end{description}

A generalized quadrangle with $s+1$ points incident with a given line and $t+1$ lines incident with a given point is said to have \textit{order} $(s,t)$, and such a generalized quadrangle is said to be \textit{thick} if both $s > 1$ and $t > 1$.    

Generalized quadrangles (and, more generally, \textit{generalized n-gons}) were invented by Jacques Tits \cite{titsngon} to help better understand certain classical groups by providing natural geometric objects on which the groups act.  The \textit{automorphism group} of a finite generalized quadrangle is the set of permutations of the point set that preserve collinearity.  While the definition of a generalized quadrangle is purely combinatorial, the known examples of generalized quadrangles all have nontrivial (and, typically, quite robust) automorphism groups; see \cite{DeWinterThas, OkeefePenttila, PayneQClan, fgq}.  For this reason, given a generalized quadrangle $\calQ$ of order $(s,t)$, it is natural to study the possible automorphisms of $\calQ$.  Toward this end, we prove the following result: 

\begin{thm}
 \label{thm:main}
 Let $\calQ$ be a thick generalized quadrangle of order $(s,t)$, where $s > t$ and $s+1$ is prime.  If $\calQ$ has an automorphism of order $s+1$, then
 \[ s \left\lceil \left\lceil \frac{t^2}{s+1}\right\rceil\left(\frac{s+1}{t} \right) \right\rceil \le t(s+t).\]
\end{thm}

If $\calQ$ is a generalized quadrangle of order $(s,t)$, where $s,t$ satisfy both the hypotheses and the inequality of Theorem \ref{thm:main}, then we cannot say much.  The real strength of this result arises in the situation where $s,t$ satisfy all of the numerical constraints of Theorem \ref{thm:main} except for the inequality, in which case we can make the following conclusion:

\begin{cor}
 \label{cor:notpointtrans}
 Let $\calQ$ be a thick generalized quadrangle of order $(s,t)$, where $s > t$ and $s+1$ is prime.  If 
 \[ s \left\lceil \left\lceil \frac{t^2}{s+1}\right\rceil\left(\frac{s+1}{t} \right) \right\rceil > t(s+t),\]
 then $\calQ$ does not have an automorphism of order $s+1$ and the automorphism group of $\calQ$ cannot be point-transitive.
\end{cor}

The \textit{dual} of a generalized quadrangle $\calQ$ with point set $\calP$ and line set $\calL$ comes from switching the role of points and lines to create a new generalized quadrangle $\calQ^\prime$ with point set $\calL$ and line set $\calP$.  Viewed through the lens of the dual quadrangle, Corollary \ref{cor:notpointtrans} can be rephrased to obtain results about the line-transitivity of certain potential generalized quadrangles.

\begin{cor}
 \label{cor:notlinetrans}
 Let $\calQ$ be a thick generalized quadrangle of order $(s,t)$, where $t > s$ and $t+1$ is prime.  If 
 \[ t \left\lceil \left\lceil \frac{s^2}{t+1}\right\rceil\left(\frac{t+1}{s} \right) \right\rceil > s(s+t),\]
 then $\calQ$ does not have an automorphism of order $t+1$ and the automorphism group of $\calQ$ cannot be line-transitive.
\end{cor}

At first glance, the inequalities listed above seem to be rather weak.  However, the following corollaries show the power of the result when one parameter is (relatively speaking) much bigger than the other.

\begin{cor}
 \label{cor:moreexplicit}
 Let $\calQ$ be a thick generalized quadrangle of order $(s,t)$.  If $s+1$ is a prime and if there exists a natural number $n$ such that
 \[ \frac{t^2}{n+1} + t \le s + 1 < \frac{t^2}{n},\]
 then $\calQ$ cannot have an automorphism of order $(s+1)$ and cannot have a point-transitive group of automorphisms.
\end{cor}

\begin{cor}
 \label{cor:q^2-q}
 If $\calQ$ is a generalized quadrangle of order $(q^2 - nq, q)$, where $n$ and $q$ are positive integers with $2n < q$ and $q^2 - nq + 1$ is prime, then $\calQ$ cannot have an automorphism of order $q^2 - nq + 1$, and, moreover, $\calQ$ does not have a point-transitive group of automorphisms.  
\end{cor}

In particular, Corollary \ref{cor:q^2-q} shows that a generalized quadrangle of order $(q^2 - q, q)$, where $q^2 - q + 1$ is prime, is not point-transitive.  Very little is known about such (potential) generalized quadrangles; see \cite{Ovoids}.  Indeed, these results, while far from definitive, make it increasingly unlikely that such generalized quadrangles exist, since they would be very asymmetric.

We are further able to use the theory that we have built up to study the automorphism groups of potential generalized quadrangles of order $(4,12)$.  The best known result thus far comes from \cite{Ovoids}, which states that if such a generalized quadrangle contains an \textit{ovoid}, a set of $st + 1$ pairwise noncollinear points, then it cannot be point-transitive.  We are able to say considerably more:

\begin{thm}
 \label{thm:412}
 If $\calQ$ is a generalized quadrangle of order $(4,12)$, then the automorphism group of $\calQ$ cannot be transitive on either points or lines.
\end{thm}

While there are certainly ``regular'' combinatorial structures that are asymmetric, Theorem \ref{thm:412} makes it much more unlikely that such a generalized quadrangle exists.  Moreover, it is likely that the techniques used in the proof of Theorem \ref{thm:412} can be used to prove that the automorphism groups of other potential generalized quadrangles cannot be point- or line-transitive.  

This paper is organized as follows.  In Section \ref{sect:background}, we provide the background material necessary for the rest of the paper.  In Section \ref{sect:autprime}, we prove various preliminary results about automorphisms of prime order of generalized quadrangles.  Section \ref{sect:proofineq} is dedicated to the proof of Theorem \ref{thm:main}, and Section \ref{sect:conseq} is dedicated to the consequences of Theorem \ref{thm:main} and, in particular, contains proofs of Corollaries \ref{cor:notpointtrans}, \ref{cor:moreexplicit}, and \ref{cor:q^2-q}.  Section \ref{sect:412} contains results that apply specifically to generalized quadrangles of order $(4,12)$, culminating in the proof of Theorem \ref{thm:412}.  Finally, we include in Appendix \ref{app:calc} some tables which list all possible orders $(s,t)$ with $t \le 100$ to which Corollary \ref{cor:notpointtrans} applies.

\section{Background}
\label{sect:background}
Let $\calQ$ be a generalized quadrangle with point set $\calP$ and line set $\calL$.  We say that two points $P,P^\prime$ are \textit{collinear} if there is a line incident with both $P$ and $P^\prime$, in which case we write $P \sim P^\prime$.  Similarly, we say that two lines $\ell, \ell^\prime$ are \textit{concurrent} if there is a point incident with both $\ell$ and $\ell^\prime$, and we write $\ell \sim \ell^\prime$.  Our convention here will be that $P \sim P$.  Given a set $X$ of points,
\[ X^\perp := \{P \in \calP : P \sim Q \text{ for all } Q \in X\}.\]
We define $Z^\perp$ for a set $Z$ of lines analogously.  If the role of ``point'' and of ``line'' (as well as the values of $s$ and $t$) are interchanged for $\calQ$, then the result is also a generalized quadrangle and is called the \textit{dual} of $\calQ$.

A \textit{grid} is an incidence structure $(\calP, \calL, \I)$ such that for some integers $s_1, s_2 \in \N$ we have
\[\calP = \{ P_{i,j} : 0 \le i \le s_1, 0 \le j \le s_2\}, \; \calL = \{\ell_0, \dots, \ell_{s_1}, \ell_0^\prime, \dots, \ell_{s_2}^\prime\}\] 
with incidence defined by $P_{i,j} \I \ell_k$ if and only if $i = k$ and $P_{i,j} \I \ell_k^\prime$ if and only if $j = k$.  A \textit{dual grid} is the point-line dual of a grid, and, instead of $s_1$ and $s_2$, it is defined in terms of parameters $t_1$ and $t_2$.  It is easy to see that a grid is a generalized quadrangle if and only if $s_1 = s_2$, and the generalized quadrangles with $t = 1$ are precisely the grids with $s_1 = s_2 (= s)$.  The dual result holds for dual grids.

The following omnibus lemma contains basic results about the parameters $s$ and $t$.

\begin{lem}\cite[1.2.1, 1.2.2, 1.2.3, 1.2.5]{fgq}
\label{lem:GQbasics}
Let $\mathcal{Q}$ be a finite generalized quadrangle of order $(s,t)$.  Then the following hold:
\begin{itemize}
 \item[(i)] $|\mathcal{P}| = (s+1)(st + 1)$ and $|\mathcal{L}| = (t+1)(st+1)$;
 \item[(ii)] $s + t$ divides $st(s+1)(t+1)$;
 \item[(iii)] if $s, t > 1$, then $t \leqslant s^2$ and $s \leqslant t^2$;
 \item[(iv)] if $1 < s < t^2$, then $s \leqslant t^2 - t$, and if $1 < t < s^2$, then $t \leqslant s^2 - s$.
\end{itemize}
\end{lem}

The following result of Payne is an application of the so-called Higman-Sims technique and is crucial to the proof of Theorem \ref{thm:main}.  A more general result is actually proved in \cite{PayneIneq}, and a proof of just Lemma \ref{lem:PayneIneq} is given in \cite[1.4.1]{fgq}.

\begin{lem}\cite[Theorem I.2]{PayneIneq}
\label{lem:PayneIneq}
 Let $X$ and $Y$ be disjoint sets of pairwise noncollinear points of a generalized quadrangle $\calQ$ of order $(s,t)$ with $s > 1$ such that $|X| = m$, $|Y| = n$, and $X \subseteq Y^\perp$.  Then $(m-1)(n-1) \le s^2$.  Dually, if $X$ and $Y$ are disjoint sets of pairwise nonconcurrent lines of a generalized quadrangle $\calQ$ of order $(s,t)$ with $t > 1$ such that $|X| = m$, $|Y| =n$, and $X \subseteq Y^\perp$, then $(m-1)(n-1) \le t^2$.
\end{lem} 

Let $x$ be an automorphism of $\calQ$.  Given a point $P$ of $\calQ$, there are three possibilities: 

\begin{itemize}
 \item[(i)] $P^x = P$,
 \item[(ii)] $P^x \neq P$ but $P^x \sim P$,
 \item[(iii)] $P^x \not\sim P$.
\end{itemize}

With this in mind, we define $\calP_0(x)$ to be the set of points fixed by $x$, $\calP_1(x)$ to be the set of points that are not fixed by $x$ but are sent to collinear points, and $\calP_2(x)$ to be the set of points sent to noncollinear points by $x$.  For each $i$, we define $\alpha_i(x) := |\calP_i(x)|$.  For the lines of $\calQ$, we define $\calL_0(x)$, $\calL_1(x)$, and $\calL_2(x)$ analogously, and for each $i$ we define $\beta_i(x) := |\calL_i(x)|$. 

The following result is known as Benson's Lemma and is a fundamental result regarding automorphisms of generalized quadrangles.

\begin{lem}\cite[Lemma 4.3]{Benson}
\label{lem:Benson}
 Let $x$ be an automorphism of a finite generalized quadrangle of order $(s,t)$.  If $\alpha_0(x)$ denotes the number of points fixed by $x$ and $\alpha_1(x)$ denotes the number of points sent to distinct noncollinear points, then
 \[(t+1)\alpha_0(x) + \alpha_1(x) \equiv (st+1) \pmod {s+t}.\]
\end{lem}

The following result relates the total number of points sent to collinear points by $x$ to the total number of lines sent to collinear lines by $x$.

\begin{lem}\cite[1.9.2]{fgq}
\label{lem:pointlinecounts}
 Let $x$ be an automorphism of a finite generalized quadrangle of order $(s,t)$.  If $\alpha_0(x)$ denotes the number of points fixed by $x$, $\alpha_1(x)$ denotes the number of points sent by $x$ to distinct noncollinear points, $\beta_0(x)$ denotes the number of lines fixed by $x$, and $\beta_1(x)$ denotes the number of lines sent by $x$ to distinct concurrent lines, then
 \[(1+t)\alpha_0(x) + \alpha_1(x) = (1+s)\beta_0(x) + \beta_1(x).\]
\end{lem}

Given an automorphism $x$ of $\calQ$, the substructure $\calQ_x$ fixed by $x$ must have one of a few types.  The following result lists these possible types.  For convenience, our delineation into types is slightly different than what is listed in \cite{fgq}.

\begin{lem}\cite[2.4.1]{fgq}
\label{lem:fixedtypes}
 Let $x$ be an automorphism of a generalized quadrangle $\calQ$.  The substructure $\calQ_x$ fixed by $x$ is one of the following:
 \begin{itemize}
  \item[(0)] The substructure $\calQ_x$ is empty; that is, there are no fixed points and there are no fixed lines.
  \item[(1)] At least one point is fixed, all fixed points are noncollinear, and no lines are fixed.
  \item[($1^\prime$)] At least one line is fixed, all fixed lines are nonconcurrent, and no points are fixed.
  \item[(2)] There exists some fixed point $P$ such that $P \sim P^\prime$ for each fixed point $P^\prime$, there exists at least one fixed line, and every fixed line is incident with $P$.
  \item[($2^\prime$)] There exists some fixed line $\ell$ such that $\ell \sim \ell^\prime$ for each fixed line $\ell^\prime$, there exists at least one fixed point, and every fixed point is incident with $\ell$.
  \item[(3)] The substructure $\calQ_x$ is a grid with $s_1 < s_2$.
  \item[($3^\prime$)] The substructure $\calQ_x$ is a dual grid with $t_1 < t_2$.
  \item[(4)] The substructure $\calQ_x$ is a generalized subquadrangle of order $(s^\prime, t^\prime)$.  
 \end{itemize}
\end{lem}

Note that we allow in (4) the possibility that $\calQ_x$ is a \textit{grid} or a \textit{dual grid}, i.e., we allow either $s^\prime = 1$ or $t^\prime = 1$.

Finally, we introduce some terminology from permutation group theory that will be used in Section \ref{sect:412}.  The action of a group $G$ on a set $\Omega$ is said to be \textit{quasiprimitive} if every nontrivial normal subgroup of $G$ is transitive on $\Omega$.  Quasiprimitive groups are a generalization of primitive permutation groups, since, if $G$ acts on $\Omega$ and $G$ contains a normal subgroup $N$ that is intransitive on $\Omega$, then the set of orbits of $N$ on $\Omega$ are a $G$-invariant partition of $\Omega$.  For a characterization of quasiprimitive permutation groups, see \cite[Section 2]{PraegerQuasiprimitive}. 

\section{Automorphisms of prime order}
\label{sect:autprime}
In this section, we collect a number of basic results about prime order automorphisms of generalized quadrangles.  Throughout this section, $\calQ$ will refer to a generalized quadrangle of order $(s,t)$ with point set $\calP$ and line set $\calL$.  For a given automorphism $x$ of $\calQ$, the \textit{type} of $\calQ_x$ refers to its designation in Lemma \ref{lem:fixedtypes}.

\begin{lem}
  \label{lem:alphaicong}
 Let $x$ be an automorphism of $\calQ$ with order $p$, where $p$ is a prime.  For $i = 1,2$, we have
 \[ \alpha_i(x), \beta_i(x) \equiv 0 \pmod p.\]
 Moreover, 
 \[\alpha_0(x) \equiv (s+1)(st+1) \pmod p\]
 and
 \[\beta_0(x) \equiv (t+1)(st+1) \pmod p.\]
\end{lem}

\begin{proof}
 The set $\calP_1(x)$ can be partioned into orbits of $\langle x \rangle$, and, since none of these points are fixed, each orbit has size $p$. This implies that $\alpha_1(x) \equiv 0 \pmod p$.  By duality, $\beta_1(x) \equiv 0 \pmod p$, and analogous arguments show that $\alpha_2(x), \beta_2(x) \equiv 0 \pmod p$.  The results for $\alpha_0(x), \beta_0(x)$ follow from \[(s+1)(st+1) = |\calP| = \alpha_0(x) + \alpha_1(x) + \alpha_2(x)\]
 and
 \[(t+1)(st+1) = |\calL| = \beta_0(x) + \beta_1(x) + \beta_2(x).\]
\end{proof}

\begin{lem}
 \label{lem:type0cong}
 Let $x$ be an automorphism of $\calQ$ of order $p$, where $p$ is a prime, and assume that $\calQ_x$ has type $(0)$.  Then either $t+1 \equiv s+1 \equiv 0 \pmod p$ or $st+1 \equiv 0 \pmod p$.  If $p$ is an odd prime, then $s+1 \equiv t+1 \equiv 0 \pmod p$ if and only if $st+1 \not\equiv 0 \pmod p.$
\end{lem}

\begin{proof}
 Since $\calQ_x$ has type $(0)$, it follows that $\alpha_0(x) = \beta_0(x) = 0$, implying that $p \mid (s+1)(st+1)$ and $p \mid (t+1)(st+1)$.  If $p \nmid (st+1)$, then $p \mid (s+1)$ and $p \mid (t+1)$ by Euclid's Lemma.  Finally, if $p$ is an odd prime and $s+1 \equiv t+1 \equiv 0 \pmod p$, then 
 \[st+1 \equiv (-1)(-1) + 1 \equiv 2 \not\equiv 0 \pmod p.\]
\end{proof}

\begin{lem}
 \label{lem:type1cong}
 Let $x$ be an automorphism of $\calQ$ of order $p$, where $p$ is a prime.  If $\calQ_x$ has type $(1)$, then $t+1 \equiv 0 \pmod p$.  If $\calQ_x$ has type $(1^\prime)$, then $s+1 \equiv 0 \pmod p$.
\end{lem}

\begin{proof}
 We will prove the result for $\calQ_x$ of type $(1)$; the analogous result for type $(1^\prime)$ follows by duality.  Since $\calQ_x$ has type $(1)$, there are no fixed lines, but there is at least one fixed point.  Let $P$ be any fixed point, and let $\calL(P)$ be the lines incident with $P$.  Since $x$ is an automorphism, if $\ell \in \calL(P)$, then $\ell^x \in \calL(P)$.  Since no line is fixed by $x$, $|\calL(P)| = t+1$, and $\calL(P)$ can be partitioned into orbits of $\langle x \rangle$, it follows that $t+1 \equiv 0 \pmod p$. 
\end{proof}

\begin{lem}
 \label{lem:type2cong}
 Let $x$ be an automorphism of $\calQ$ of order $p$, where $p$ is a prime.  
 \begin{enumerate}[(i)]
   \item If $\calQ_x$ has type $(2)$ and $\alpha_0(x) = 1$, then $s+1 \equiv 1 \pmod p$.  
   \item If $\calQ_x$ has type $(2)$ and $\alpha_0(x) \ge 2$, then $t+1 \equiv 1 \pmod p$.
   \item If $\calQ_x$ has type $(2^\prime)$ and $\beta_0(x) = 1$, then $t+1 \equiv 1 \pmod p$.
   \item If $\calQ_x$ has type $(2^\prime)$ and $\beta_0(x) \ge 2$, then $s+1 \equiv 1 \pmod p$.
 \end{enumerate}
\end{lem}

\begin{proof}
 We will prove the results for $\calQ_x$ of type $(2)$ and note that the analogous results for type $(2^\prime)$ follow by duality.  Assume first that $\alpha_0(x) = 1$, and let $P$ be this unique fixed point.  By assumption, there is at least one fixed line $\ell$ incident with $P$, and the $s$ remaining points of $\ell$ are partitioned into orbits of size $p$ of $\langle x \rangle$, proving that $s+1 \equiv 1 \pmod p$.
 \par Now assume that $\alpha_0(x) \ge 2$. Let $P,Q$ be two distinct fixed points, where we may assume by hypothesis that $P\sim P'$ for every $P'\in\calP_0(x).$ Hence $P \sim Q$, and so $x$ also fixes the unique line $\ell$ with which $P$ and $Q$ are mutually incident.  Since none of the other $t$ lines incident with $Q$ are fixed by $x$, these $t$ lines are partitioned into orbits of size $p$, proving that $t+1 \equiv 1 \pmod p$. 
\end{proof}

\begin{lem}
 \label{lem:type2fixed}
 Let $x$ be an automorphism of $\calQ$ of order $p$, where $p$ is a prime.  If $\calQ_x$ has type $(2)$, then
 \[\alpha_0(x) \equiv 1 + s\beta_0(x) \pmod p.\]
 If $\calQ_x$ has type $(2^\prime)$, then
 \[\beta_0(x) \equiv 1 + t\alpha_0(x) \pmod p.\]
\end{lem}

\begin{proof}
 We will prove the result for $\calQ_x$ of type $(2)$; the result when $\calQ_x$ is of type $(2^\prime)$ follows by duality.  Let $P$ be the distinguished point with which all fixed lines of $x$ are incident and all fixed points of $x$ are collinear.  For any fixed line $\ell$, let $s_0(\ell)$ be the number of fixed points incident with $\ell$ other than $P$ and let $s_1(\ell)$ be the number of points incident with $\ell$ not fixed by $x$.  Noting that the $s_1(\ell)$ points of $\ell$ that are not fixed by $x$ are partitioned into orbits of size $p$ of $\langle x \rangle$, we have $s_1(\ell) \equiv 0 \pmod p$, which implies that $s_0(\ell) \equiv s \pmod p$ since $s_0(\ell) + s_1(\ell) = s$.  If the $\beta_0(x)$ lines fixed by $x$ are $\ell_1, \dots, \ell_{\beta_0(x)}$, then
 \[\alpha_0(x) = 1 + \sum_{i=1}^{\beta_0(x)} s_0(\ell_i) \equiv 1 + \sum_{i=1}^{\beta_0(x)} s \pmod p,\]
 and so 
 \[\alpha_0(x) \equiv 1 + s\beta_0(x) \pmod p,\]
 as desired.
\end{proof}

\begin{lem}
 \label{lem:type3fixed}
Let $x$ be an automorphism of $\calQ$ of order $p$, where $p$ is a prime.  If $\calQ_x$ has type $(3)$, then $t+1 \equiv 2 \pmod p$ and $s_1 \equiv s_2 \equiv s \pmod p$.  If $\calQ_x$ has type $(3^\prime)$, then $s+1 \equiv 2 \pmod p$ and $t_1 \equiv t_2 \equiv t \pmod p$.  In particular, if $\calQ_x$ has type $(3)$ or type $(3^\prime)$, then  $p < \max\{s,t\}$, and, if $\calQ$ is a thick generalized quadrangle and $\calQ_x$ has either type (3) or type ($3^\prime$), then $p < \min\{s,t\}$.   
\end{lem}

\begin{proof}
 We will prove the result for $\calQ_x$ of type $(3)$; the result when $\calQ_x$ is of type $(3^\prime)$ follows by duality.  Let $P$ be a fixed point of the grid.  Since exactly two lines incident with $P$ are fixed by $x$, the remaining lines incident with $P$ must be partitioned into $\langle x \rangle$-orbits of size $p$, and hence $t + 1 \equiv 2 \pmod p$.  Now, there are two types of lines in the grid: those containing $s_1 + 1$ fixed points and those containing $s_2 + 1$ fixed points.  For a line $\ell$ of $\calQ$ fixed by $x$ containing $s_1 + 1$ fixed points in $\calQ_x$, the remaining $(s+1) - (s_1 + 1)$ points incident with $\ell$ are partioned into $\langle x \rangle$-orbits, and so $s_1 \equiv s \pmod p$.  Analogously, we have $s_2 \equiv s \pmod p$.  Finally, the prime $p$ divides $(t + 1) - 2 = t- 1$, so $p \le t - 1$ if $t > 1$, and, since $s_1 < s_2 \le s$, $p$ divides $s - s_1 > 0$ and $p \le s - s_1$. The result follows.
\end{proof}

\begin{lem}
 \label{lem:neither1modp}
 Let $x$ be an automorphism of $\calQ$ of order $p$, where $p$ is a prime.  If $s+1 \not\equiv 0,1,2 \pmod p$ and $t+1 \not\equiv 0,1,2 \pmod p$, then either $\calQ_x$ has type $(0)$ and $st+1 \equiv 0 \pmod p$, or $\calQ_x$ has type $(4)$.
\end{lem}

\begin{proof}
 We will proceed through the types listed in Lemma \ref{lem:fixedtypes}.  If $\calQ_x$ has type $(0)$, then since $p$ divides neither $s+1$ nor $t+1$, it follows that $st+1 \equiv 0 \pmod p$ by Lemma \ref{lem:type0cong}.  If $\calQ_x$ has type $(1)$ or $(1^\prime)$, then either $s+1 \equiv 0 \pmod p$ or $t+1 \equiv 0 \pmod p$ by Lemma \ref{lem:type1cong}, contrary to our hypotheses.  If $\calQ_x$ has type $(2)$ or $(2^\prime)$, then either $s+1 \equiv 1 \pmod p$ or $t+1 \equiv 1 \pmod p$ by Lemma \ref{lem:type2cong}, a contradiction to our hypotheses.  Finally, if $\calQ_x$ has type $(3)$ or type $(3^\prime)$, then either $s+1 \equiv 2 \pmod p$ or $t+1 \equiv 2 \pmod p$ by Lemma \ref{lem:type3fixed}, again a contradiction. 
\end{proof}

\begin{lem}
 \label{lem:type4cong}
 Let $x$ be an automorphism of $\calQ$ of order $p$, where $p$ is a prime.  If $\calQ_x$ has type $(4)$ and is a subquadrangle of order $(s^\prime, t^\prime)$, then $s^\prime \equiv s \pmod p$ and $t^\prime \equiv t \pmod p$.
\end{lem}

\begin{proof}
 Let $\ell$ be any line fixed by $x$.  By hypothesis, there are exactly $s^\prime + 1$ points fixed by $x$ on $\ell$, and hence there are $(s+1) - (s^\prime + 1) = (s - s^\prime)$ points on $\ell$ that are not fixed by $x$.  These remaining $(s - s^\prime)$ points are partitioned into orbits of $\langle x \rangle$ of size $p$, and hence $s^\prime \equiv s \pmod p$.  By duality, $t^\prime \equiv t \pmod p$.
\end{proof}

\begin{lem}
 \label{lem:type4values}
 Let $x$ be an automorphism of $\mathcal{Q}$ of order $p$, where $p$ is a prime.  If $\calQ_x$ has type $(4)$ and is a proper  subquadrangle of order $(s, t^\prime)$, then
 \begin{align*}
  \alpha_0(x) &= (s+1)(st^\prime + 1) \\
  \alpha_1(x) &= 0 \\
  \alpha_2(x) &= s(s+1)(t - t^\prime)\\
 \end{align*}
and
\begin{align*}
 \beta_0(x) &= (t^\prime + 1)(st^\prime + 1) \\
 \beta_1(x) &= (t - t^\prime)(s+1)(st^\prime + 1) \\
 \beta_2(x) &= (t+1)(st + 1) - (t(s+1) - st^\prime + 1)(st^\prime + 1). \\
\end{align*}
\end{lem}

\begin{proof}
 First, $\alpha_0(x) = (s+1)(st^\prime + 1)$, since $\calQ_x$ is a subquadrangle of order $(s, t^\prime)$.  Similarly, $\beta_0(x) = (t^\prime + 1)(st^\prime + 1)$.  We will now show that $\alpha_1(x) = 0$.  Let $P$ be a point that is not fixed by $x$.  If $\ell$ is a line fixed by $x$, then, since $\calQ_x$ is a subquadrangle of order $(s, t^\prime)$, all points incident with $\ell$ are fixed by $x$. This means that $P$ is not incident with $\ell$, and so, by the GQ Axiom, there exists a unique point $Q$ on $\ell$ with which $P$ is collinear.  Let $\ell^\prime$ be the line incident with both $P$ and $Q$.  Since $Q$ is incident with $\ell$, $Q$ is fixed.  The line $\ell^\prime$ cannot be fixed, since $P$ is not fixed by $x$.  However, $(\ell^\prime)^x$ is also incident with $Q$, and, by the GQ Axiom, there are no triangles, which means $P \not\sim P^x$ and $\alpha_1(x) = 0$.  The value of $\beta_1(x)$ now follows from Lemma \ref{lem:pointlinecounts}, and the values of $\alpha_2(x)$ and $\beta_2(x)$ follow from
 \[|\calP| = \alpha_0(x) + \alpha_1(x) + \alpha_2(x)\]
 and
 \[|\calL| = \beta_0(x) + \beta_1(x) + \beta_2(x),\]
 respectively.
\end{proof}

\begin{lem}
 \label{lem:type4s=sprime}
 Let $x$ be an automorphism of $\mathcal{Q}$ of order $p$, where $p$ is a prime.  If $\calQ_x$ has type $(4)$ and $s^\prime = s$, then $s+t$ divides $st^\prime(st+1)$.
\end{lem}

\begin{proof}
 By Lemmas \ref{lem:Benson} and \ref{lem:type4values},
 \[(t+1)(s+1)(st^\prime +1) \equiv st + 1 \pmod {s+t}.\]
 The result follows after simplification of this expression.
\end{proof}
 
\begin{lem}
 \label{lem:type4largep}
 Let $p$ be a prime such that $p \ge s$, and suppose $x$ is an automorphism of $\calQ$ of order $p$ such that $\calQ_x$ has type $(4)$.  Then $s^\prime = s$, $t^\prime < t$, $t^\prime \equiv t \pmod p$, and $s+t$ divides $st^\prime (st+1)$.
\end{lem}

\begin{proof}
 By Lemma \ref{lem:type4cong}, $s^\prime \equiv s \pmod p$, and, since $p \ge s \ge s^\prime$, we have $s^\prime = s$.  The result now follows from Lemmas \ref{lem:type4cong} and \ref{lem:type4s=sprime}.
\end{proof}

\begin{lem}
 \label{lem:s+1prime}
 Let $\calQ$ be a generalized quadrangle of order $(s,t)$, where $s > t$ and $s+1$ is prime.  If $x$ is an automorphism of $\calQ$ of order $s+1$, then $\calQ_x$ has type $(1^\prime)$.
\end{lem}

\begin{proof}
 Assume first that $\calQ_x$ has type $(0)$.  Since $s > t$, $(s+1) \nmid (t+1)$, and hence $st + 1 \equiv 0 \pmod {s+1}$ by Lemma \ref{lem:type0cong}.  However, this implies that 
 \[s(t-1) \equiv (st+1) - (s+1) \equiv 0 \pmod {s+1}.\]
 By Euclid's Lemma, either $(s+1) \mid s$ or $(s+1) \mid (t-1)$, which is impossible since $s+1 > s, t-1$.  Hence $\calQ_x$ cannot have type $(0)$.\\
 If $\calQ_x$ has type $(1)$, then $t+1 \equiv 0 \pmod {s+1}$ by Lemma \ref{lem:type1cong}, which is impossible since $s > t$.  If $\calQ_x$ has either type $(2)$ or type $(2^\prime)$, then either $s+1 \equiv 1 \pmod {s+1}$ or $t+1 \equiv 1 \pmod {s+1}$ by Lemma \ref{lem:type2cong}, again a contradiction.  If $\calQ_x$ has either type $(3)$ or type $(3^\prime)$, then either $s+1 \equiv 2 \pmod {s+1}$ or $t+1 \equiv 2 \pmod {s+1}$ by Lemma \ref{lem:type3fixed}, another contradiction.  Finally, if $\calQ_x$ has type $(4)$, then by Lemma \ref{lem:type4largep} we have $s = s^\prime$ and $(s+1)$ divides $(t - t^\prime)$.  However, $t - t^\prime$ is both smaller than $s+1$ and nonzero, a contradiction.
\end{proof}

\begin{lem}
 \label{lem:primerestriction}
 Let $p$ be a prime that divides the order of the automorphism group of a finite generalized quadrangle $\calQ$ of order $(s,t)$.  If $p \nmid (st+1)$, then $p \le \max\{s+1, t+1\}$.
\end{lem}

\begin{proof}
 First, if $t = 1$, then $\calQ$ is a grid with automorphism group isomorphic to the wreath product $\Sym(s+1) \Wr 2$, and so $p \le s + 1$, with an analogous result holding in the dual grid case.  Hence we may assume that $\calQ$ is a thick generalized quadrangle.   Assume $p > (s+1)$ and $p > (t+1)$.  Let $x$ be an automorphism of $\calQ$ of order $p$.  Since $\calQ$ is thick, we have $s + 1 \not\equiv 0, 1, 2 \pmod p$ and $t + 1 \not\equiv 0, 1, 2 \pmod p$.  By Lemma \ref{lem:neither1modp}, if $p \nmid (st+1)$, then $\calQ_x$ has type $(4)$.  However, since $p > s+1$, by Lemma \ref{lem:type4largep} we have that $s = s^\prime$ and $p \mid (t - t^\prime)$, a contradiction since $p > s+1 > t - t^\prime > 0$.  The result follows.
\end{proof}

It should be noted that this last lemma yields a definitive list of primes that could be orders of automorphisms of a generalized quadrangle $\calQ$ without knowing any information about $\calQ$ other than its order $(s,t)$.

\section{Proof of the inequality}
\label{sect:proofineq}

This section is devoted to the proof of Theorem \ref{thm:main}.

\begin{proof}[Proof of Theorem \ref{thm:main}]
 Let $\calQ$ be a thick generalized quadrangle of order $(s,t)$, where $s > t$ and $s+1$ is prime, and let $\calQ$ have an automorphism $x$ of order $s+1$.  By Lemma \ref{lem:s+1prime}, $\calQ_x$ has type $(1^\prime)$, thus no points are fixed by $x$ and at least one line is fixed by $x$.  Let $\ell$ be a line fixed by $x$.  Since $\calQ_x$ has type $(1^\prime)$, no points incident with $\ell$ are fixed and all fixed lines are pairwise nonconcurrent, and so the lines concurrent with $\ell$ are divided into $t$ distinct orbits of $\langle x \rangle$ of size $s+1$.  If $\ell^\prime$ is any other fixed line, then $|\{\ell, \ell^\prime\}^\perp| = s + 1$, i.e., there is a unique $\langle x \rangle$-orbit of lines concurrent with $\ell$ that is also concurrent with $\ell^\prime$.  By the Pigeonhole Principle, one of the $t$ different $\langle x \rangle$-orbits of lines concurrent with $\ell$, which we name $X$, is also concurrent with at least $\left\lceil (\beta_0(x) - 1)/t \right\rceil$ fixed lines other than $\ell$.  If $Y$ is the set of $\left\lceil (\beta_0(x) - 1)/t \right\rceil + 1$ lines that are all nonconcurrent, fixed by $x$, and incident with each line in $X$, then, by Lemma \ref{lem:PayneIneq},
 \[s \left\lceil \frac{\beta_0(x) - 1}{t} \right\rceil \le t^2.\] 
To finish the proof, we provide a lower bound on $\beta_0(x)$.  Since $s+1$ is prime, by Lemma \ref{lem:alphaicong} we have
\[ \beta_0(x) \equiv (t+1)(st+1) \pmod {s+1},\] which equivalently implies that
\[ \beta_0(x) \equiv -(t^2 - 1) \pmod {s+1}.\]  Thus, there exists some $k \in \N$ such that 
\[\beta_0(x) = k(s+1) - (t^2 - 1).\]
If $k < t^2/(s+1)$, then
\[ \beta_0(x) = k(s+1) - (t^2 - 1) < \frac{t^2}{s+1}(s+1) - (t^2 - 1) = 1,\]
which implies that $\beta_0(x) < 1$, a contradiction, since $\calQ_x$ has type $(1^\prime)$.  Thus $k \ge \left\lceil \frac{t^2}{s+1} \right\rceil,$ and so 
\[ \left\lceil \frac{t^2}{s+1}\right\rceil (s+1) - (t^2 - 1) \le \beta_0(x).\]  This means
\[s\left( \left\lceil \left\lceil\frac{t^2}{s+1} \right\rceil \cdot \frac{s+1}{t}\right\rceil -t\right) = s\left\lceil \frac{\left(\left\lceil\frac{t^2}{s+1} \right\rceil(s+1) - (t^2- 1)\right) - 1}{t}\right\rceil \le s\left\lceil \frac{\beta_0(x) - 1}{t} \right\rceil \le t^2.\]
Simplifying, we have
 \[ s \left\lceil \left\lceil \frac{t^2}{s+1}\right\rceil\left(\frac{s+1}{t} \right) \right\rceil \le t(s+t),\]
 as desired.
\end{proof}

\section{Consequences of the inequality}
\label{sect:conseq}

In this section, we present some consequences of Theorem \ref{thm:main}.  First, we have immediately Corollary \ref{cor:notpointtrans}, which says that a generalized quadrangle of order $(s,t)$, where $s > t$, $s+1$ is prime, and 
\[ s \left\lceil \left\lceil \frac{t^2}{s+1}\right\rceil\left(\frac{s+1}{t} \right) \right\rceil > t(s+t),\]
is not point-transitive.

\begin{proof}[Proof of Corollary \ref{cor:notpointtrans}]
Assume that $\calQ$ has order $(s,t)$, where $s > t > 1$ and $s+1$ is prime.  If $\calQ$ has an automorphism group $G$ that is transitive on points and $P$ is a point of $\calQ$, then
\[ |G| = |\calP||G_P| = (s+1)(st+1)|G_P|.\]
Since the prime $(s+1)$ divides $|G|$, $G$ must have an element of order $s+1$, which means $s$ and $t$ must satisfy the hypotheses of Theorem \ref{thm:main}.  The result follows.
\end{proof}

We can also now prove Corollary \ref{cor:notlinetrans}.

\begin{proof}[Proof of Corollary \ref{cor:notlinetrans}]
 This follows immediately from point-line duality and Corollary \ref{cor:notpointtrans}.
\end{proof}

At first glance, the inequality 
\[s \left\lceil \left\lceil \frac{t^2}{s+1}\right\rceil\left(\frac{s+1}{t} \right) \right\rceil \le t(s+t)\]
does not seem like much of a restriction.  However, as we will see, when $s$ is much larger than $t$, there are often situations when the ceiling functions contained in the inequality make a drastic difference.  Corollary \ref{cor:moreexplicit} shows that one implication of Theorem \ref{thm:main} is that, if $s+1$ is prime and if there exists a natural number $n$ such that
\[ \frac{t^2}{n+1} + t \le s + 1 < \frac{t^2}{n},\]
then a generalized quadrangle of order $(s,t)$ cannot have an automorphism of order $s+1$ and cannot be point-transitive. 

\begin{proof}[Proof of Corollary \ref{cor:moreexplicit}]
 Assume that $s+1$ is prime and that 
 \[ \frac{t^2}{n+1} + t \le s + 1 < \frac{t^2}{n}\]
 for some natural number $n$.  First, since 
 \[ t < \frac{t^2}{n+1} + t \le s + 1,\]
 $s \ge t$.  If $s = t$, then $t^2/(n+1) \le 1$, which implies that 
 \[ s + 1 < \frac{t^2}{n} = \left(\frac{t^2}{n+1}\right)\left(\frac{n+1}{n}\right) \le 1 + \frac{1}{n},\]
 a contradiction.  Hence $s > t$.  On the other hand, since 
  \[ \frac{t^2}{n+1} < \frac{t^2}{n+1} + t \le s + 1 < \frac{t^2}{n},\]
  we have  
  \[ n < \frac{t^2}{s+1} < n+1,\]
  and so  
  \[ \left\lceil \frac{t^2}{s+1} \right\rceil = n + 1.\]
  Moreover, since $t^2/(n+1) + t \le s + 1$
  \[ (n+1)(s+1) \ge t^2 + (n+1)t,\] and, since $t > 1$,
  \[ (n+1)s \ge t^2 + (n+1)(t-1) > t^2.\]
  Thus,        
  \begin{align*}
   s \left\lceil \left\lceil \frac{t^2}{s+1}\right\rceil\left(\frac{s+1}{t} \right) \right\rceil &= \left\lceil (n+1)\frac{s+1}{t} \right\rceil \\
   &\ge s \left\lceil \frac{t^2 + (n+1)t}{t} \right\rceil\\
   &= s(t + (n+1))\\
   &= ts + (n+1)s\\
   &> ts + t^2\\
   &= t(s + t),
  \end{align*}
and, by Theorem \ref{thm:main} and Corollary \ref{cor:notpointtrans}, such a generalized quadrangle cannot have an automorphism of order $s+1$ and cannot be point-transitive.
\end{proof}

One particular application of the inequality is Corollary \ref{cor:q^2-q}, which states that, if $\calQ$ is a generalized quadrangle of order $(q^2 - nq, q)$ and $q^2 - nq + 1$ is prime with $q > 2n$, then $\calQ$ is not point-transitive.  These conditions apply to numerous potential generalized quadrangles whose existence is not known, for instance orders $(12,4)$, $(30, 6)$, $(42, 7)$, and $(72, 9)$; see Appendix \ref{app:calc} for many more instances.

\begin{proof}[Proof of Corollary \ref{cor:q^2-q}]
 Let $\calQ$ be a generalized quadrangle of order $(q^2 - nq, q)$, where $q > 2n$ and $q^2 - nq + 1$ is prime.  In this instance,
 \[ \left\lceil \frac{q^2}{q^2 - nq + 1} \right\rceil = 2,\]
 and 
 \[ \left\lceil 2 \frac{q^2 - nq + 1}{q} \right\rceil = \left\lceil 2q - 2n + \frac{2}{q} \right\rceil = 2q - 2n + 1,\]
 and so
 \begin{align*}
  (q^2 - q) \left\lceil \left\lceil \frac{q^2}{q^2 - q + 1} \right\rceil \left(\frac{q^2 - q + 1}{q}\right) \right\rceil &= (q^2 - q)(2q-2n +1)\\
  &= q\left((q-1)(2q-2n+1)\right)\\
  &> q \cdot q^2\\
  &> q\left((q^2 - nq) + q \right)
 \end{align*}
when $q > 2n$, and hence by Theorem \ref{thm:main} and Corollary \ref{cor:notpointtrans}, such a generalized quadrangle cannot have an automorphism of order $q^2 - nq + 1$ and cannot be point-transitive.
\end{proof}

It is unknown whether $q^2 - nq + 1$ is prime for infinitely many positive integer values of $q$ for a fixed $n$.  However, the following conjecture from number theory supports this conclusion.

\begin{conj}\cite{Bouniakowsky}
 Let $f(x) = a_nx^n + \cdots + a_1x + a_0$ be a polynomial with integer coefficients.  The set $\{k \in \mathbb{Z} : f(k) \text{ is prime}\}$ is infinite if the following three conditions hold:
 \begin{enumerate}[(i)]
  \item $a_n = 1$,
  \item $f$ is irreducible over $\mathbb{Z}$,
  \item The set of integers $f(\mathbb{Z}) = \{f(n) : n \in \mathbb{Z}\}$ has greatest common divisor $1$.
 \end{enumerate}
\end{conj}

For $f(x) = x^2 - nx + 1$, it is plain to see that $f$ satisfies conditions (i) and (ii) when $n \neq 2$, and $f(n) = 1$, showing (iii).  The numerical evidence in Appendix \ref{app:calc} lends evidence that there could indeed be infinitely many such pairs where $(s,t) = (q^2 - nq, q)$ that satisfy $s+t \mid st(st+1)$ where $s+1$ is prime.  

It is an interesting question as to whether generalized quadrangles of such orders actually exist.  While the asymmetry of such examples is potential evidence against existence, combinatorial regularity also does not necessitate symmetry. 

\section{Automorphisms of a generalized quadrangle of order (4,12)}
\label{sect:412}

This section is dedicated to analyzing the structure of the automorphism group of a generalized quadrangle of order $(4,12)$, if one were to exist.  Throughout this section, $\calQ$ will be a generalized quadrangle of order $(4,12)$ with point set $\calP$, line set $\calL$, and automorphism group $G$.  As in the previous sections, for $x \in G$, the \textit{type} of $\calQ_x$ refers to its designation under Lemma \ref{lem:fixedtypes}.

\begin{lem}
\label{lem:412ple7}
 If $p$ is a prime that divides $|G|$, then $p \le 7$.
\end{lem}

\begin{proof}
 Let $p$ be a prime dividing $|G|$.  By Lemma \ref{lem:primerestriction}, $p \le 13$.  We know that no automorphism of order $13$ exists by Corollary \ref{cor:notlinetrans}, and so we assume $p = 11$ and let $x$ be an element of $G$ of order $11$.  By Lemma \ref{lem:type0cong}, $\calQ_x$ cannot be of type $(0)$; by Lemma \ref{lem:type1cong}, $\calQ_x$ cannot be of type $(1)$ or type $(1^\prime)$; by Lemma \ref{lem:type2cong}, $\calQ_x$ cannot be of type $(2)$ or $(2^\prime)$; by Lemma \ref{lem:type3fixed}, $\calQ_x$ cannot be of type $(3)$ or $(3^\prime)$; and, by Lemma \ref{lem:type4largep}, $\calQ_x$ cannot be of type $(4)$.  Therefore, if $p$ divides $|G|$, then $p \le 7$.
\end{proof}

\begin{lem}
 \label{lem:412p=7}
 If $x \in G$ is an element of order $7$, then $\alpha_0(x)  = 0$.
\end{lem}

\begin{proof}
 Let $x$ be an element of $G$ of order $7$.  By Lemma \ref{lem:type1cong}, $\calQ_x$ cannot be of type $(1)$ or of type $(1^\prime)$.  By Lemma \ref{lem:type2cong}, $\calQ_x$ cannot be of type $(2)$ or of type $(2^\prime)$.  By Lemma \ref{lem:type3fixed}, $\calQ_x$ cannot be of type $(3)$ or of type $(3^\prime)$ By Lemma \ref{lem:type4largep}, $\calQ_x$ cannot be of type $(4)$.  Therefore, $\calQ_x$ is of type $(0)$ and $\alpha_0(x) = 0$.
\end{proof}

\begin{lem}
 \label{lem:412syl7}
 A Sylow $7$-subgroup of $G$ has order at most $49$.
\end{lem}

\begin{proof}
 Let $X$ be a Sylow $7$-subgroup of $G$, and let $y \in X$.  If $y$ is not the identity and $y$ fixes any point of $\calQ$, then $y^{|y|/7}$ is an element of order $7$ that fixes a point of $\calQ$, a contradiction to Lemma \ref{lem:412p=7}.  This implies that $X$ acts semiregularly on $\mathcal{P}$, and so $|X|$ divides $|\calP| = 245$. The result follows.  
\end{proof}

\begin{lem}
 \label{lem:412p=5}
 If $h \in G$ is an element of order $5$, then $\alpha_0(h) = 0$.
\end{lem}

\begin{proof}
  Let $h$ be an element of $G$ of order $5$.  By Lemma \ref{lem:type0cong}, $\calQ_h$ cannot be of type $(0)$.  By Lemma \ref{lem:type1cong}, $\calQ_h$ cannot be of type $(1)$.  By Lemma \ref{lem:type2cong}, $\calQ_h$ cannot be of type $(2)$ or of type $(2^\prime)$. By Lemma \ref{lem:type3fixed}, $\calQ_h$ cannot be of type $(3)$ or of type $(3^\prime)$.  By Lemma \ref{lem:type4largep}, $\calQ_h$ cannot be of type $(4)$.  Therefore, $\calQ_h$ is of type $(1^\prime)$ and $\alpha_0(h) = 0$.
\end{proof}

\begin{lem}
 \label{lem:412syl5}
 A Sylow $5$-subgroup of $G$ has order at most $5$.
\end{lem}

\begin{proof}
 Let $H$ be a Sylow $5$-subgroup of $G$, and let $y \in H$.  If $y$ is not the identity and $y$ fixes any point of $\calQ$, then $y^{|y|/5}$ is an element of order $5$ that fixes a point of $\calQ$, a contradiction to Lemma \ref{lem:412p=5}.  This implies that $H$ acts semiregularly on $\mathcal{P}$, and so $|H|$ divides $|\calP| = 245$. The result follows.
\end{proof}

\begin{lem}
\label{lem:412imprim}
 If $G$ is transitive on $\calP$, then the action of $G$ on $\calP$ is not quasiprimitive, i.e., $G$ must contain a nontrivial normal subgroup that is intransitive on $\calP$.
\end{lem}

\begin{proof}
 Assume that the action of $G$ on $\calP$ is quasiprimitive.  By \cite[Theorem 1]{PraegerQuasiprimitive}, since $|\calP|$ is not a prime power, $G$ must have a nonabelian minimal normal subgroup $N = T^k$, where $T$ is a nonabelian finite simple group and $k \in \N$, such that $N$ is transitive on $\calP$.  Moreover, by Lemma \ref{lem:412ple7}, no prime greater than $7$ divides $N$.  Assume $k \ge 2$.  Since the largest power of $5$ to divide $G$ is $5$, in this case $5 \nmid |T|$.  However, since $N$ is transitive on $\calP$, $5$ divides $|N|$, and so $5$ divides $|T|$, a contradiction. Hence $N = T$ is a finite nonabelian simple group.
 
 On the other hand, the only primes that can divide $|T|$ are $2,3,5,7$.  Moreover, $5$ and $7$ all must divide $|T|$, since $T$ is transitive on $\calP$, and the largest power of $5$ dividing $|T|$ is $5$ and the largest power of $7$ dividing $|T|$ is $49$.  By \cite[Theorem II]{HuppertLempken}, there is no such finite simple group.  Hence the action of $G$ on $\calP$ cannot be quasiprimitive.
\end{proof}

\begin{lem}
 \label{lem:412order35}
 If $G$ is transitive on $\calP$, then $G$ contains an element of order $35$.
\end{lem}

\begin{proof}
 Assume that $G$ is transitive on $\calP$.  It suffices to show that either the normalizer of a $5$-subgroup contains an element of order $7$ or the normalizer of a $7$-subgroup contains an element of order $5$, since, in either case, the normalizing element is forced to be in the centralizer of the $p$-subgroup.
 
 Since $G$ is transitive but not quasiprimitive on $\calP$ by Lemma \ref{lem:412imprim}, $G$ must contain an intransitive normal subgroup $N$.  Let $P$ be a Sylow $p$-subgroup of $N$ for some prime $p$.  By the Frattini Argument (see, for instance, \cite[Theorem 1.13]{Isaacs}), $G = N_G(P) N$.  This means that $|G|$ divides $|N_G(P)| \cdot |N|$.
 
 Since $N$ is intransitive on $\calP$, there are four possibilities: (i) there are $5$ distinct $N$-orbits of size $49$, (ii) there are $7$ distinct $N$-orbits of size $35$, (iii) there are $35$ distinct $N$-orbits of size $7$, or (iv) there are $49$ distinct $N$-orbits of size $5$.  

 Consider first the case when there are $5$ distinct $N$-orbits of size $49$.  Since $N$ is transitive on a set of size $49$, $49 \mid |N|$.  Let $P$ be a Sylow $7$-subgroup of $N$, which has size $49$.  Since $G$ is transitive on the five $N$-orbits, $5 \mid |G:N|$.  Since $G$ is not divisible by $25$, this implies that $5 \nmid |N|$.  However, since $5$ divides $|G|$, $|G|$ divides $|N_G(P)| \cdot |N|$, and $5$ does not divide $|N|$, we have that $5$ divides $|N_G(P)|$, and so $G$ contains an element of order $5$ that normalizes (and hence centralizes) a Sylow $7$-subgroup of $G$.
 
 We proceed similarly in the remaining cases: if there are $7$ distinct $N$-orbits of size $35$, then $7$ divides $|N_G(P)|$, where $P$ is a Sylow $5$-subgroup of $N$; if there are $35$ distinct $N$-orbits of size $7$, then $5$ divides $|N_G(P)|$, where $P$ is a Sylow $7$-subgroup of $N$; and, finally, if there are $49$ distinct $N$-orbits of size $5$, then $7$ divides $|N_G(P)|$, where $P$ is a Sylow $5$-subgroup of $N$.  In any case, if $G$ is transitive on $\calP$, then $G$ must contain an element of order $35$, as desired.
\end{proof}

We are now ready to prove Theorem \ref{thm:412}.

\begin{proof}[Proof of Theorem \ref{thm:412}]
Let $\calQ$ be a generalized quadrangle of order $(4,12)$, and let $G = \Aut(\calQ)$.  By Lemma \ref{lem:412ple7}, $13 \nmid |G|$, and so $G$ cannot be transitive on lines.

Assume that $G$ is transitive on points.  By Lemma \ref{lem:412order35}, $G$ must contain an element $h$ of order $35$.  Since $|h^5| = 7$, by Lemma \ref{lem:412p=7}, $\alpha_0(h) = 0$.  Consider $\calP_1(h)$, the set of points sent to distinct collinear points by $h$.  The orbits of $\langle h \rangle$ have size $5$, $7$, or $35$, and $\calP_1(h)$ is made up of these orbits.  However, since $|h^5| = 7$ and $|h^7| = 5$, both $h^5$ and $h^7$ are semiregular on $\calP$ by Lemmas \ref{lem:412p=7} and \ref{lem:412p=5}, and so no orbit of $\langle h \rangle$ can have size $5$ or $7$.  Thus $\alpha_1(h) \equiv 0 \pmod {35}$.

On the other hand, by Benson's Lemma (Lemma \ref{lem:Benson}), $\alpha_1(h) \equiv 1 \pmod {16}$.  By the Chinese Remainder Theorem, this means that $\alpha_1(h) \equiv 385 \pmod {560}$.  Since $\alpha_1(h) \le |\calP| = 245$, we reach a contradiction, and so $G$ cannot be transitive on $\calP$, as desired.
\end{proof}

Finally, we remark that, while the techniques used in this section relied heavily on the exact values of $s$ and $t$, the ideas used here should be applicable to other relatively small values of $s$ and $t$.

\bibliographystyle{plain}
\bibliography{GQPrimeAutomorphism}

\newpage

\appendix
\section{Calculations}
\label{app:calc}

As $t$ increases, there seems to be a steady increase in the proportion of feasible parameters $(s,t)$ of generalized quadrangles that satisfy the hypotheses of Corollary \ref{cor:notpointtrans} and hence cannot be point-transitive if they exist. 

Now, we give tables of all possible orders $(s,t)$ of generalized quadrangles with $t \le 100$ that satisfy the hypotheses of Corollary \ref{cor:notpointtrans}. The tag $(***)$ denotes that this order $(s,t)$ has the form $s = t^2 - nt$ where $s+1$ is prime and $2n<t$.
\begin{table}[H]\centering
    \begin{tabular}{ | l | l | l |l |}
    \hline
    (12, 4) ***   & (312, 26)      & (946, 44)      & (826, 59)      \\ \hline
    (22, 6)       & (442, 26) ***  & (1276, 44) *** & (660, 60)      \\ \hline
    (30, 6) ***   & (540, 27) ***  & (1408, 44) *** & (672, 60)      \\ \hline
    (42, 7) ***   & (378, 28)      & (576, 45)      & (1038, 60)     \\ \hline
    (28, 8)       & (756, 28) ***  & (630, 45)      & (1740, 60)     \\ \hline
    (40, 8) ***   & (270, 30)      & (990, 45)      & (2136, 60)     \\ \hline
    (36, 9)       & (280, 30)      & (1012, 46)     & (2380, 60)     \\ \hline
    (72, 9) ***   & (420, 30)      & (456, 48)      & (3540, 60) *** \\ \hline
    (40, 10)      & (232, 32)      & (540, 48)      & (1830, 61)     \\ \hline
    (60, 12)      & (672, 32) ***  & (1128, 48)     & (1860, 62)     \\ \hline
    (66, 12)      & (330, 33)      & (1296, 48) *** & (2542, 62) *** \\ \hline
    (78, 13)      & (442, 34)      & (1666, 49) *** & (3906, 63) *** \\ \hline
    (156, 13) *** & (1122, 34) *** & (460, 50)      & (576, 64)      \\ \hline
    (126, 14) *** & (280, 35)      & (700, 50)      & (768, 64)      \\ \hline
    (210, 15) *** & (490, 35)      & (970, 50)      & (976, 64)      \\ \hline
    (112, 16)     & (700, 35) ***  & (1200, 50)     & (1216, 64)     \\ \hline
    (240, 16) *** & (396, 36)      & (1450, 50) *** & (1600, 64)     \\ \hline
    (136, 17)     & (408, 36)      & (612, 51)      & (2016, 64)     \\ \hline
    (96, 18)      & (556, 36)      & (2550, 51) *** & (760, 65)      \\ \hline
    (210, 18)     & (612, 36)      & (796, 52)      & (910, 65)      \\ \hline
    (306, 18) *** & (630, 36)      & (1300, 52)     & (2080, 65)     \\ \hline
    (130, 20)     & (852, 36)      & (1326, 52)     & (3510, 65) *** \\ \hline
    (148, 20)     & (456, 38)      & (540, 54)      & (1408, 66)     \\ \hline
    (180, 20)     & (546, 39)      & (918, 54)      & (2112, 66)     \\ \hline
    (190, 20)     & (1482, 39) *** & (936, 54)      & (2346, 66)     \\ \hline
    (280, 20) *** & (616, 40)      & (1566, 54) *** & (4422, 67) *** \\ \hline
    (316, 20)     & (760, 40)      & (2376, 54) *** & (1666, 68)     \\ \hline
    (330, 20)     & (820, 41)      & (1870, 55) *** & (3060, 68) *** \\ \hline
    (126, 21)     & (546, 42)      & (2970, 55) *** & (1380, 69)     \\ \hline
    (210, 21)     & (732, 42)      & (616, 56)      & (2346, 69)     \\ \hline
    (420, 21) *** & (1162, 42)     & (742, 56)      & (910, 70)      \\ \hline
    (462, 22) *** & (1722, 42) *** & (856, 56)      & (2380, 70)     \\ \hline
    (136, 24)     & (430, 43)      & (1008, 56)     & (4830, 70) *** \\ \hline
    (276, 24)     & (316, 44)      & (1288, 56)     & (1096, 72)     \\ \hline
    (336, 24) *** & (616, 44)      & (2296, 56) *** & (1656, 72)     \\ \hline
    (456, 24) *** & (676, 44)      & (1596, 57)     & (2520, 72)     \\ \hline
    (600, 25) *** & (682, 44)      & (3306, 58) *** & (2556, 72)     \\ \hline
     \end{tabular}
\end{table}

\newpage

\begin{table}[H]\centering
    \begin{tabular}{ | l | l | l |}
    \hline
(3432, 72)            &  (1092, 84)  & (3060, 90)     \\ \hline
(5112, 72) ***        & (1276, 84)     & (3186, 90)     \\ \hline
(1776, 74)           & (1582, 84)   & (3690, 90)      \\ \hline
(1050, 75)       & (1596, 84)     & (5580, 90) ***    \\ \hline
(1800, 75)       & (1876, 84)   & (6210, 90) ***     \\ \hline
(4200, 75) ***  & (2268, 84)      & (8010, 90) ***     \\ \hline
(1596, 76)      & (2296, 84)     & (2002, 91)   \\ \hline
(2052, 76)      & (2436, 84)    & (3796, 91) \\ \hline
(2850, 76)      & (2856, 84)     & (8190, 91) ***    \\ \hline
(5700, 76) ***   & (4200, 84) ***  & (3082, 92) \\ \hline
(4642, 77)       & (4326, 84)     &     (6256, 92) ***    \\ \hline
(936, 78)           & (4956, 84) ***      &  (2790, 93)    \\ \hline
(1950, 78)      & (6036, 84)     &   (1692, 94)      \\ \hline
(2766, 78)        & (6580, 84)     &   (4512, 96)   \\ \hline
 (6006, 78) ***  & (990, 85)   & (4560, 96)   \\ \hline
 (6162, 79) *** & (1360, 85)     &  (6112, 96) \\ \hline
 (880, 80)      & (3570, 85)     &   (4656, 97)  \\ \hline
(1200, 80)     & (3612, 86)    &  (1288, 98)   \\ \hline
 (1216, 80)     & (4902, 86) ***      & (3136, 98)   \\ \hline
 (1360, 80)     & (2436, 87)     & (4018, 98) \\ \hline
 (1720, 80)     & (1870, 88)    & (2376, 99)  \\ \hline
 (2080, 80)     & (3916, 89)     & (2926, 99)  \\ \hline
 (2620, 80)     & (1530, 90)     & (3168, 99)   \\ \hline
 (2800, 80)      & (1548, 90)     &  (4356, 99)  \\ \hline
 (3120, 80)     & (1800, 90)   & (5346, 99) *** \\ \hline
 (3760, 80) *** & (1860, 90)     & (1900, 100) *** \\ \hline
 (4240, 80) ***  & (2010, 90)    & (4950, 100)  \\ \hline
 (4240, 80) *** & (2016, 90)     & (9900, 100) ** \\ \hline
  (4720, 80) *** & (2178, 90)    &  \\ \hline
(6480, 81) *** & (2250, 90)    &  \\ \hline

    \end{tabular}
\end{table}

\end{document}